\documentclass[letterpaper, 10 pt, conference]{ieeeconf}

\IEEEoverridecommandlockouts

\usepackage{color}
\usepackage{enumerate,graphicx,epstopdf}
\usepackage{amsmath,amssymb,bm,amsthm}
\usepackage{cite}
\usepackage{mathtools}
\usepackage{array}
\usepackage{algpseudocode,algorithm,algorithmicx}
\usepackage{color}
\usepackage{booktabs}
\usepackage[hyphens]{url}
\usepackage{lipsum}
\usepackage[dvipsnames]{xcolor}
\usepackage[deletedmarkup=sout,authormarkup=superscript]{changes}
\usepackage[normalem]{ulem}

\usepackage{mwe}
\newtheorem{ass}{Assumption}
\newtheorem{rmk}{Remark}

\newtheorem{lmm}{Lemma}

\newtheorem{thm}{Theorem}

\newcommand{\ints}{\mathbb{N}}
\newcommand\intrng[2]{\ints_{[#1,#2]}}
\newcommand{\reals}{\mathbb{R}}

\newcommand{\mc}[1]{\mathcal{#1}}

\newcommand{\proj}[1]{\Pi_{#1}}
\newcommand{\Int}[0]{\mathrm{Int}~}
\renewcommand{\ker}[0]{\mathrm{Ker}~}

\DeclareMathOperator*{\argmin}{arg\,min}

\definechangesauthor[name={Marco}, color=NavyBlue]{MN}
\definechangesauthor[name={Dominic}, color=RedViolet]{DLM}
\definechangesauthor[name={Terry}, color=OliveGreen]{TS}
\definechangesauthor[name={Torbjørn}, color=Purple]{TC}

\begin{document}
\title{Feasibility Governor for Linear Model Predictive Control
\thanks{ M. Nicotra and T. Skibik are with the University of Colorado, Boulder, Email: \texttt{\{marco.nicotra, terrence.skibik\}@colorado.edu}.}
\thanks{D. Liao-McPherson, T. Cunis and I. Kolmanovsky are with the University of Michigan, Ann Arbor. Email:\texttt{\{dliaomcp, tcunis, ilya\}@umich.edu}.}  \thanks{This research is supported by the National Science Foundation Award Numbers CMMI 1904441, CMMI 1904394, and the Toyota Research Institute (TRI). TRI provided funds to assist the authors with their research but this article solely reflects the opinions and conclusions of its authors and not TRI or any other Toyota entity.}}
\author{Terrence Skibik, Dominic Liao-McPherson, Torbjørn Cunis, Ilya Kolmanovsky, and Marco M. Nicotra}

\maketitle

\begin{abstract}
This paper introduces the Feasibility Governor (FG): an add-on unit that enlarges the region of attraction of Model Predictive Control by manipulating the reference to ensure that the underlying optimal control problem remains feasible. The FG is developed for linear systems subject to polyhedral state and input constraints. Offline computations using polyhedral projection algorithms are used to construct the feasibility set. Online implementation relies on the solution of a convex quadratic program that guarantees recursive feasibility. The closed-loop system is shown to satisfy constraints, achieve asymptotic stability, and exhibit zero-offset tracking. 
\end{abstract}


\section{Introduction}

Model Predictive Control \cite{rawlings2018model,goodwin2006constrained,mayne2014model} (MPC) is a feedback policy that computes the solution of a receding horizon Optimal Control Problem (OCP) at every sampling instant. 
A common approach for guaranteeing the stability of MPC is to impose suitable conditions on the final step of the OCP \cite{mayne2000mpc,chen1998quasi}. The Region of Attraction (ROA) of the resulting closed-loop system is then given by  all the states that can reach the terminal constraint set within the prediction horizon.

Since the terminal set is centered on the desired reference, sudden reference changes can cause the OCP to become infeasible if the  system is unable to reach the new terminal set within the prediction horizon. Although this issue could be avoided by increasing the prediction horizon,  doing so can significantly increase the computational complexity of the controller.

A different option for increasing the ROA is to treat aspects of the terminal set as optimization variables and use the additional degrees of freedom to enlarge the feasible set. This approach has been applied to regulation \cite{gonzalez2009enlarging}, and reference tracking \cite{limon2008mpc,simon2014reference} of linear systems, and also economic operation of nonlinear systems \cite{fagiano2013generalized}. Alternatively, \cite{limon2005enlarging} enlarges the ROA by computing a contractive sequence of terminal sets offline and incorporating them into the OCP. The drawback of all these methods is that they rely on a non-standard OCP, making them difficult to combine with other MPC schemes.

\begin{figure}
	\centering
	\includegraphics[width=\columnwidth]{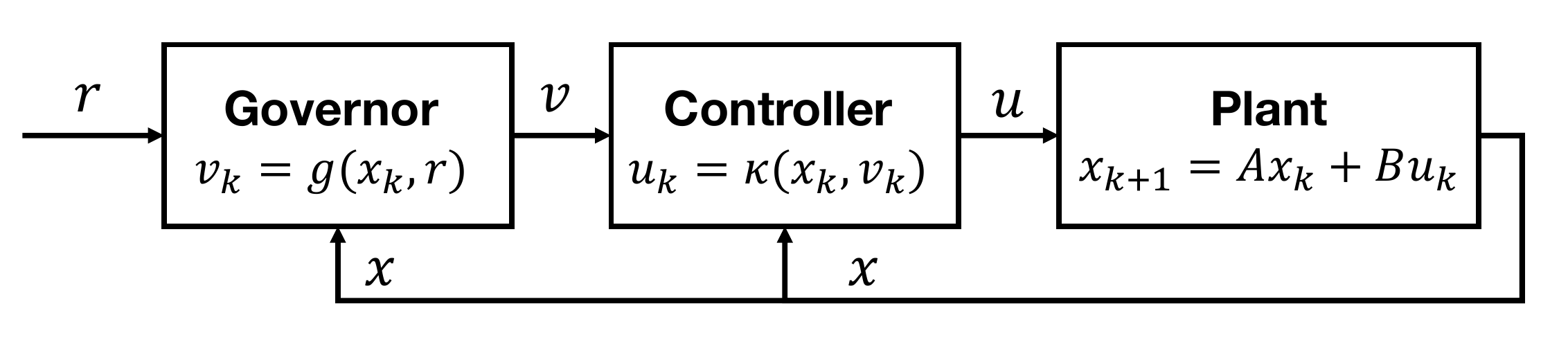}
	\caption{A block diagram of the control architecture. Given a reference $r$, the Feasibility Governor manipulates the auxiliary reference $v$ to ensure that the primary MPC controller is able to produce a valid control input $u$.}
	\label{fig:fg_block_diagram}
\end{figure}

In this paper, we introduce the Feasibility Governor (FG), an add-on unit that filters the reference signal to ensure that the terminal set remains reachable within the prediction horizon. The FG takes advantage of offline polyhedral set manipulation tools \cite{herceg2013multi,lohne2017vector} to reduce online complexity while minimizing conservatism. Doing so enables the FG to expand the ROA to the set of initial conditions that can reach the terminal set of \emph{any} steady state admissible reference, as opposed to just the target reference. This is achieved without any modifications to the existing MPC controller. Moreover, we prove that the FG ensures the constraints are never violated, and exhibits finite time convergence to the desired reference. The proposed control architecture is illustrated in Figure~\ref{fig:fg_block_diagram}.

The idea of manipulating the reference to avoid infeasibility in MPC can be found in prior literature. The recovery mode featured in \cite{chisci2003dual} simultaneously computes a modified reference and control input to enforce feasibility at the expense of performance. An FG-like algorithm that combines a governor and explicit MPC controller into a single unit is proposed in \cite{olaru2005compact}, but suffers from the complexity limitations of explicit MPC \cite{bemporad2002explicit}. A suboptimal continuous-time analog of the FG 
is proposed in \cite{DEMPC_tac}. Finally, a spatial governor specifically designed for precision machining applications is proposed in \cite{di2018cascaded}. This paper provides both a detailed look into the theoretical properties of the proposed FG, and a method for computing the feasible sets.

\textit{Notation:} For vectors $a$ and $b$, $(a,b) = [a^T~~b^T]^T$. The identity matrix is $I_N \in \reals^{N\times N}$. 
Given $M\in \reals^{m\times n}$, $\ker M = \{x~|~Mx = 0\}$. Given $x \in \reals^n$ and a positive definite matrix $P \in \reals^{n \times n}$, the weighted norm is $\|x\|_P = \sqrt{x^TPx}$. 
Given $x\in \reals^n$, $y\in \reals^m$ and a set $\Gamma \subseteq \reals^{n+m}$, the projection of $\Gamma$ onto the domain of $x$ is $\Pi_x \Gamma$, where $\Pi_x = [I_{n}~0_{n\times m}]$, and 
the slice (or cross-section) operation is $S_y(\Gamma,x) = \{y~|~(x,y)\in \Gamma\}$. 

\section{Problem Setting}
Consider the linear time invariant (LTI) system 
\begin{subequations} \label{eq:LTI_system}
\begin{align} 
x_{k+1} &= A x_k + B u_k\\
y_k &= Cx_k + D u_k\\
z_k &= E x_k + F u_k,\end{align}
\end{subequations}
where $k\in \ints$ is the discrete-time index and $x_k \in \reals^{n_x}$, $u_k \in \reals^{n_u}$, $y_k \in \reals^{n_y}$, and $z_k \in \reals^{n_z}$ are the states, control inputs, constrained outputs, and tracking outputs, respectively. System \eqref{eq:LTI_system} is subject to pointwise-in-time constraints
\begin{equation}\label{eq:constraints}
   y_k \in \mc{Y},\quad  \forall k \in \ints,
\end{equation}
where $\mc{Y}\subseteq \reals^{n_y}$ is the constraint set.

\begin{ass} \label{ass:stabilizable}
The pair $(A,B)$ is stabilizable.
\end{ass}
As detailed in \cite{limon2008mpc}, Assumption~\ref{ass:stabilizable} implies that
\begin{equation} \label{eq:equilibria}
    Z = \begin{bmatrix}
		A- I_{n_x} & B & 0\\
		E & F & -I_{n_z}
	\end{bmatrix}
\end{equation}
satisfies $\ker\!\!(Z)\neq\{0\}$ and, as a result, system \eqref{eq:LTI_system} admits a family of equilibrium points satisfying $Z\zeta = 0$ with $\zeta = (x,u,z) \neq 0$. Moreover, it is possible to introduce an auxiliary reference vector $v\in \reals^{n_v}$ to parameterize the equilibrium manifold as $\bar{\zeta}_v = (\bar{x}_v, \bar{u}_v, \bar{z}_v) = Gv$ where
\begin{equation}
   G^T := \left[G_x^T~~G_u^T~~G_z^T\right] 
\end{equation}
is a basis for $\ker\!\!(Z)$. The following assumption ensures that there is a one-to-one correspondence between the reference $v$ and the tracking output $z$.
\begin{ass} \label{eq:G_invert}
The matrix $G_z$ is invertible.
\end{ass}
Under Assumption~\ref{eq:G_invert}, it is possible to impose $G_z=I_{n_z}$ using the change of basis $G\leftarrow GG_z^{-1}$.

\begin{ass}\label{ass:constraint_set}
The set $\mc{Y}$ is a compact polyhedron with representation $\mc{Y} = \{y~|~Y y \leq h\}$ and satisfies $0 \in \Int \mc{Y}$.
\end{ass}
Given the design parameter $\epsilon \in (0, 1)$ and the corresponding set of strictly steady-state admissible references
\begin{equation} \label{eq:Veps}
   \mc{R}_\epsilon = \{v~|~(CG_x + DG_u) v \in (1-\epsilon)\mc{Y}\},
\end{equation}
we now state the control problem addressed by this paper.\medskip

\noindent\textbf{Control Objectives:} Given the LTI system \eqref{eq:LTI_system} subject to constraints \eqref{eq:constraints}, let $r\in\reals^{n_z}$ be a target reference. The goal of this paper is to design a full state feedback law that achieves
\begin{itemize}
    \item \textit{Safety:} $y_k\in\mathcal{Y},\quad \forall k \geq 0$;
    \item \textit{Convergence:} $\lim_{k\to\infty}z_k= r^\star$, where 
    \begin{equation*}
      r^\star = \argmin_{v\in \mc{R}_\epsilon}~\|v-r\|.
    \end{equation*}
\end{itemize}

\begin{rmk}
When the tracking problem is well posed, i.e., $r\in \mc{R}_\epsilon$, we recover $\lim_{k\to\infty}z_k=r$.
\end{rmk}

\section{Control Strategy}
Due the constraints, we approach the control objectives using a typical MPC formulation where the feedback policy is defined using the solution to the following OCP
\begin{subequations} \label{eq:LMPC_OCP}
\begin{alignat}{2}
\underset{\mu}{\mathrm{min}}& &&||\xi_N - \bar x_v||_P^2 + \sum_{i=0}^{N-1} ||\xi_i - \bar x_v||_Q^2 + ||\mu_i - \bar u_v||_R^2\\
\mathrm{s.t.}& ~ &&~\xi_0 = x, \label{eq:ocp_cstr1} \\
& &&~\xi_{i+1} = A\xi_i + B \mu_i, ~~~ i \in \intrng{0}{N-1},\\
& &&~C \xi_i + D \mu_i \in \mc{Y},\label{eq:ocp_cstr2} ~~~~~~ i \in \intrng{0}{N-1},\\
& &&\qquad(\xi_N,v) \in \mc{T}, \label{eq:ocp_cstr3}
\end{alignat}
\end{subequations}
where $N\in \ints_{> 0}$ is the prediction horizon, $\mu = (\mu_0,\ldots \mu_{N-1})$, 
$P$, $Q$, and $R$ are weighting matrices, and $\mc{T}\subseteq\reals^{n_x}\times\reals^{n_v}$ is the terminal set, which is assumed to be polyhedral, i.e.,
\begin{equation}
	\mc{T} = \{(x,v)~|~ T_x x + T_v v \leq c\}.
\end{equation}
The following assumptions ensure that \eqref{eq:LMPC_OCP} is well-posed and can be used to construct a stabilizing feedback law.

\begin{ass}\label{ass:DARE}
The stage cost matrices satisfy $Q \succeq 0$, with $(A,Q)$ detectable, and $R \succ 0$.
\end{ass}

Given Assumption~\ref{ass:DARE}, let $P$ be the solution to the discrete algebraic Riccati equation 
\begin{equation}
    P = Q+A^TPA-(A^TPB)(R+B^TPB)^{-1}(B^T PA),
\end{equation}
let $K$ be the associated LQR gain \begin{equation}
    K = (R+B^TPB)^{-1}(B^T PA),
\end{equation} and let $\mc{T}=\tilde{O}_{\infty}^{\epsilon}$ be the terminal set, with $\tilde{O}_{\infty}^{\epsilon}$ defined in \cite{gilbert1991linear}. 
By construction, the terminal set $\mc{T}$ is invariant and constraint admissible, i.e., $(x,v)\in\mc{T}$ implies 
\begin{subequations}\label{eq:terminal_set}
\begin{align}
    \left(A-BK\right)x+B\left(\bar{u}_v+K\bar{x}_v\right)&\in\mathcal{X}(v), \\
    \left(C-DK\right)x+D\left(\bar{u}_v+K\bar{x}_v\right)&\in\mathcal{Y},
\end{align}
\end{subequations}
where $\mathcal{X}(v) = S_x(\mc{T},v)$. We now have all the elements typically used to define an asymptotically stable MPC feedback policy \cite{mayne2000mpc}. However, the control action can be computed only if  \eqref{eq:LMPC_OCP} admits a solution. The set of all parameters for which the OCP admits a solution, i.e., the feasible set, is
\begin{equation}
    \label{eq:feasible_set}
	\Gamma_N = \{(x,v)~|~\exists~\mu:~ \eqref{eq:ocp_cstr1}-\eqref{eq:ocp_cstr3}\} \subseteq \reals^{n_x}\times\reals^{n_v},
\end{equation}  
which is the $N$-step backwards reachable set of $\mc{T}$. Since $\mc{T}$ is polyhedral, then $\Gamma_N$ is also polyhedral and can be computed offline, as detailed in Section~\ref{ss:feasible_set_computation}.
Assuming $(x,v)\in\Gamma_N$, it is possible to compute the MPC feedback policy
\begin{equation} \label{eq:MPC_feedback}
	\kappa(x,v) = \mu^\star_0(x,v)
\end{equation}
where $\mu^\star(x,v)=[\mu_0^{\star T},\mu_1^{\star T},\ldots,\mu_{N-1}^{\star T}]^T$ is the minimizer of \eqref{eq:LMPC_OCP}. The following theorem summarizes the properties of the closed-loop system for a constant auxiliary reference.
\begin{thm} \label{thm:MPC_stab}
Let Assumptions~\ref{ass:stabilizable}--\ref{ass:DARE} hold and let $\phi(\ell,x,v)$ denote the solution of the closed-loop dynamics
\begin{equation} \label{eq:closed_loop_dynamics}
x_{k+1} = f(x_k,v) \coloneqq Ax_k + B\kappa(x_k,v).
\end{equation} 
starting from the initial condition $x_0 = x$ at timestep $\ell \geq 0$.
Then for all $(x,v) \in \Gamma_N$.
\begin{itemize}
    \item $(\phi(\ell,x,v),v)\in\Gamma_N,~\forall \ell\geq 0$;
    \item $y_\ell \in \mc{Y},~\forall \ell\geq 0$;
    \item $\lim_{\ell\to\infty} \phi(\ell,x,v) = \bar x_{v}$.
\end{itemize}
If, in addition, $v \in \mc{R}_\epsilon$ then $\bar x_{v}$ is asymptotically stable.
\end{thm}
\begin{proof}
Since the auxiliary reference $v$ is constant for  $\ell\geq 0$, the statement follows from \cite[Theorem 4.4.2]{goodwin2006constrained}.
\end{proof}

Theorem \ref{thm:MPC_stab} achieves the control objectives given $v=r$ and $x_0$ satisfying $(x_0,r)\in\Gamma_N$ with an ROA of $S_x(\Gamma_N,r)$. Its main limitation, however, lies in the fact that the OCP \eqref{eq:LMPC_OCP} is infeasible if $x_0$ cannot be steered to $\mathcal{X}(r)$ within $N$ steps. Although increasing the prediction horizon may seem like a suitable workaround, this solution may be inapplicable under real-time restrictions since the computational time required to solve \eqref{eq:LMPC_OCP} scales unfavorably with $N$. 
\begin{figure}
	\centering
	\includegraphics[width = \columnwidth]{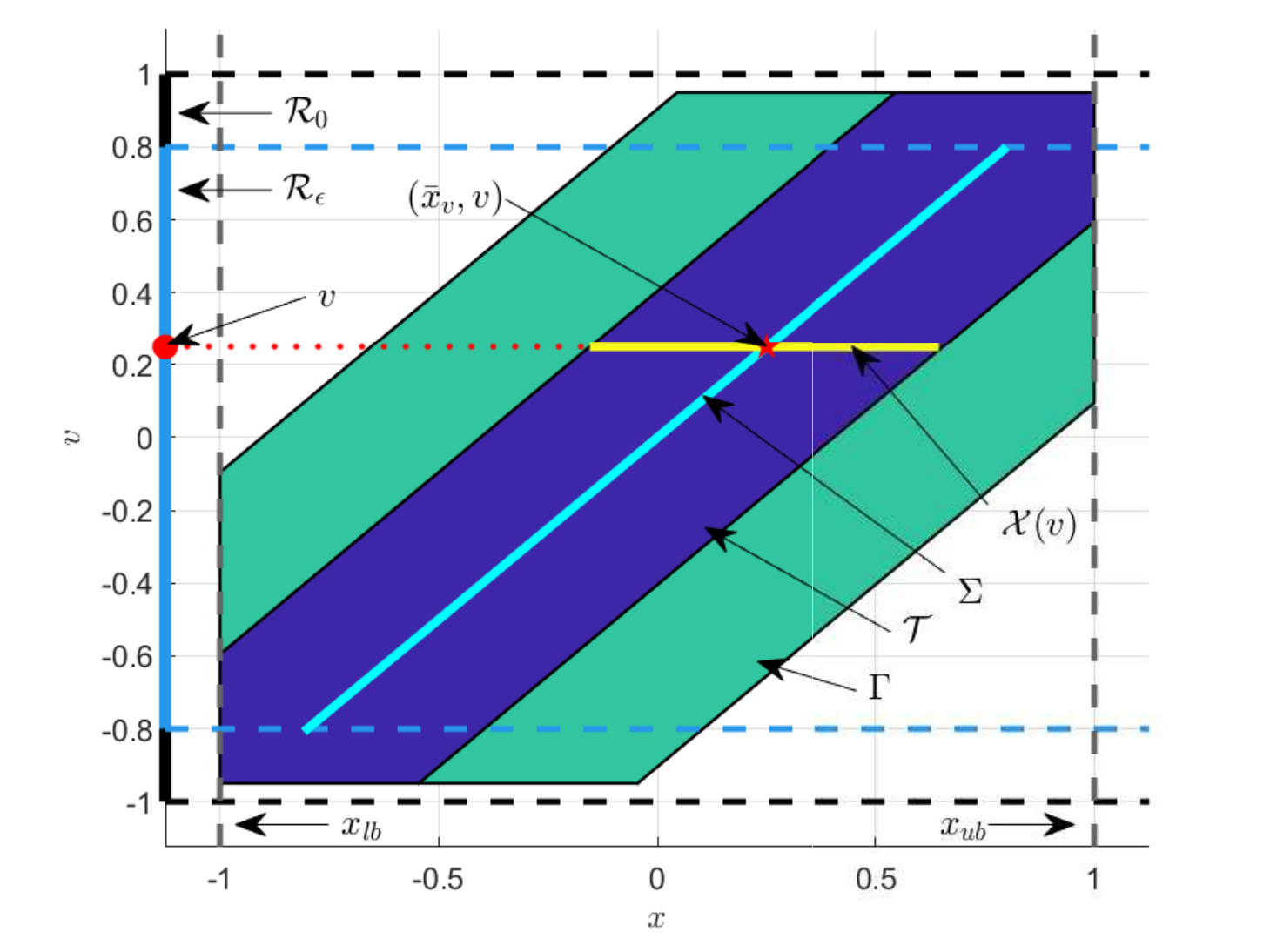}
	\caption{Illustration of the sets used in the paper for the discrete-time integrator $x_{k+1} = x_k + u_k$ subject to $|x_k| \leq 1$ and $|u_k|\leq 0.25$, given $\epsilon = 0.2$, $\mathcal{T} = \tilde{O}_{\infty}^{0.05}$ and $N=2$.}
	\label{fig:Lemma1(1)}
\end{figure} 
In view of extending the ROA, we define the set of strictly steady-state admissible equilibria
\begin{equation}\label{eq:Sigma}
   \Sigma = \{(x,v)~|~
	 x=G_x v,~ v \in\mc{R}_\epsilon \},
\end{equation}
and note the following.
\begin{lmm}\label{lmm:strictly_ss}
$\Sigma \subset \Int\Gamma_N$.
\end{lmm}
\begin{proof}
Since $(\bar x_v,v) \in \Int \tilde{O}_\infty^\epsilon$ \cite[Theorem 2.1]{gilbert1991linear}, the result follows from $\Sigma\subset\Int\mc{T}\subseteq\Int\Gamma_N$
\end{proof}

Figure~\ref{fig:Lemma1(1)} depicts all the sets defined in this section. In the next section, we describe an \emph{add-on} unit that expands the closed-loop domain of attraction without extending the prediction horizon or modifying the MPC formulation.

\section{The Feasibility Governor}  \label{ss:fg_design}
The MPC feedback policy \eqref{eq:MPC_feedback} is stabilizing only if the terminal set associated with the target equilibrium is $N$-step reachable from the current state. Intuitively, this limitation can be overcome by selecting a sequence of intermediate targets that are pair-wise reachable. This paper formalizes this idea by redefining the auxiliary reference $v$ as a time-varying system $v_k$ to ensure $(x_k,v_k)\in\Gamma_N,~\forall k\in\ints$, and $v_{k}=r$ for sufficiently large $k\in\ints$. The resulting control architecture is displayed in Figure~\ref{fig:fg_block_diagram}.

\subsection{Governor Design}
The idea behind the FG is straightfoward: modify the target reference as little as needed to ensure that the OCP remains feasible. Drawing inspiration from the Command Governor (CG) literature \cite{garone2017reference,bemporad1997nonlinear}, the FG policy can be computed via by solving
\begin{equation} \label{eq:fg_opt_problem} 
	g(x,r) = \argmin_{v\in \mc{R}_\epsilon}~\left\{ \|v - r\|_2^2~|~(x,v) \in \Gamma_N\right\}.
\end{equation}

Given a measurement $x_k$, the FG computes a virtual reference $v_k = g(x_k,r)$ that is passed to the MPC controller to obtain a control action $u_k = \kappa(x_k,v_k)$. 

\subsection{Properties}

Given system \eqref{eq:LTI_system} and the feedback policy \eqref{eq:MPC_feedback}, the FG is recursively feasible (see Theorem~\ref{thm:rec_feas}), guarantees constraint satisfaction (see Theorem~\ref{thm:rec_feas}), renders the point $x^*_r = G_x v^*_r$ asymptotically stable (see Theorem~\ref{thm:asymptotic_convergence}), and exhibits finite time convergence of $v_k \to v^*_r$ (see Theorem~\ref{thm:finite_time_convergence}).

Moreover, the FG expands the ROA of the closed-loop system from $\mc{D}_{MPC} = S_x(\Gamma_N,r^\star)$, i.e., the set of states from which it is possible to reach $\mc{X}(r^\star)$ in $N$-steps, to
\begin{equation}
	\mc{D}_{FG} = \bigcup_{v\in \mc{R}_\epsilon} S_x(\Gamma_N,v),
\end{equation}
i.e., the set of states from which it is possible to reach $\mc{X}(v)$ for \textit{any} $v\in \mc{R}_\epsilon$. In particular, the addition of the FG guarantees safe transitions between any $r_1,r_2 \in \mc{R}_\epsilon$. The differences between $\mc{D}_{MPC}$ and $\mc{D}_{FG}$ are illustrated in Figure~\ref{fig:2D_traj} for the double integrator example in Section~\ref{ss:double_integrator}.

\begin{figure}
	\centering
	\includegraphics[width = \columnwidth]{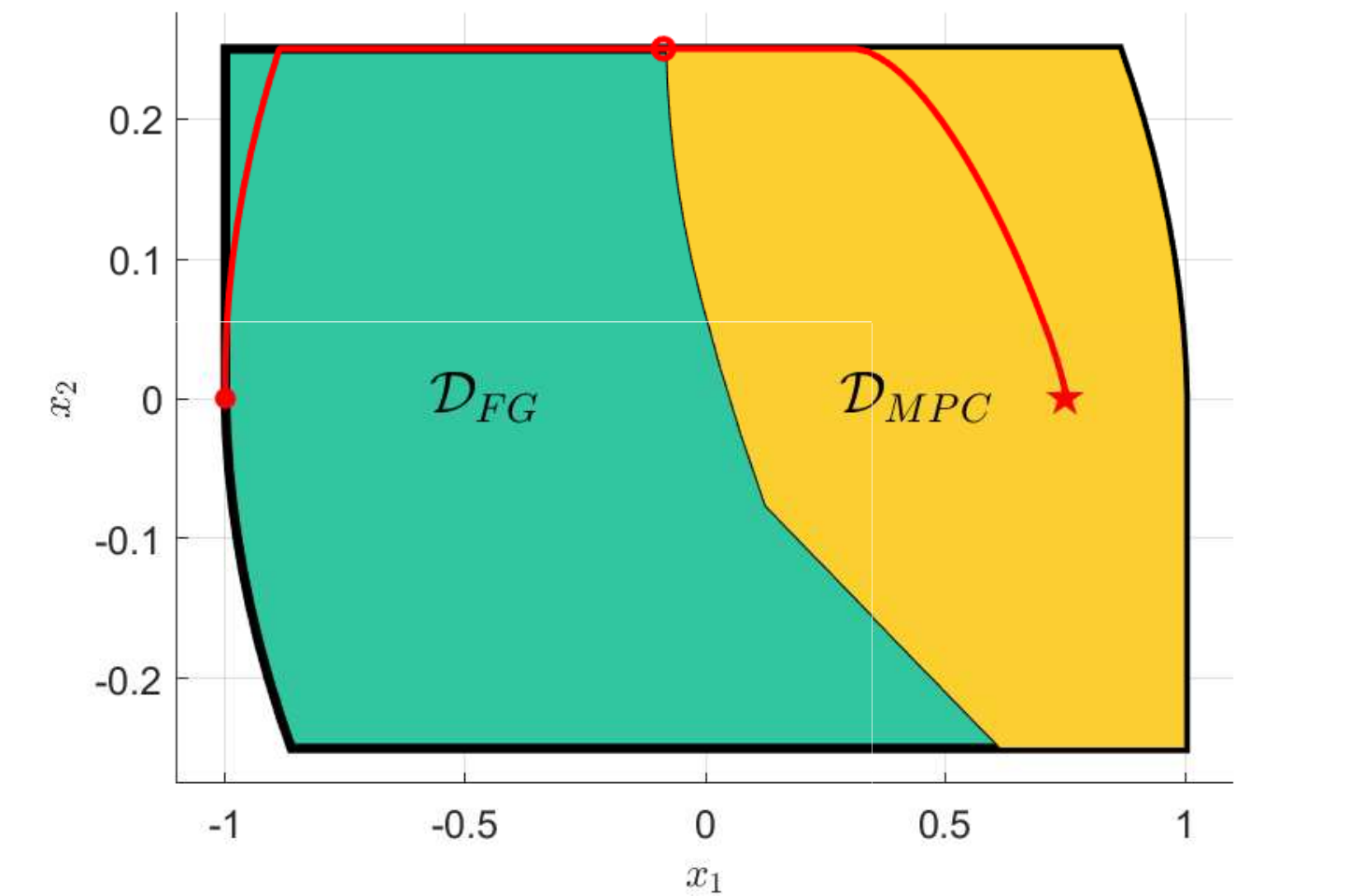}
	\caption{For the double integrator example in Section~\ref{ss:double_integrator}, the region of attraction of the combined MPC + FG feedback law (green + yellow) is significantly larger than that of the MPC controller alone (yellow), given a prediction horizon of $N=10$. The red trajectory is from the double integrator example in Section \ref{ss:double_integrator}.}
	\label{fig:2D_traj}
\end{figure}

\subsection{Implementation} \label{ss:feasible_set_computation}
Since $\Gamma_N$ and $\mc{R}_\epsilon$ are polyhedral, \eqref{eq:fg_opt_problem} is a strongly convex quadratic program (QP), and can therefore be solved in real-time. 
Moreover, since the FG problem typically has a small number of variables and many constraints, dual active-set methods \cite{goldfarb1983numerically} are particularly well suited for solving \eqref{eq:fg_opt_problem} efficiently and reliably due to the limited number of active constraints at any given time.

Implementation of the FG also requires a half-space representation of the feasible set $\Gamma_N$. To compute such a representation, note that the OCP \eqref{eq:LMPC_OCP} is a QP and can be written in the condensed form \cite{liaomc2019fbstab}
\begin{subequations} \label{eq:Condensed_LMPC_OCP}
\begin{alignat}{2} 
\underset{\mu}{\mathrm{min.}}&\quad \frac12&& \mu^T H \mu + \mu^T W\theta\\
\mathrm{s.t.}& && M \mu + L \theta \leq b,
\end{alignat}
\end{subequations}
with parameter $\theta = (x,v)$. The feasible set \eqref{eq:feasible_set} can therefore be expressed as
\begin{equation} \label{eq:gamma_block}
	\Gamma_N = \proj{\theta} \{(\mu,\theta)~|~ M \mu + L \theta \leq b\}.
\end{equation}

Several toolboxes are available for performing polyhedral calculus (e.g., projections, images, inverse images etc.). In this paper, we compute $\Gamma_N$ using the \texttt{bensolve tools} \cite{ciripoi2018calculus} package. Unfortunately, the complexity of computing $\Gamma_N$ is dominated by the projection operation. The projection is performed offline but can quickly become intractable even for moderately sized systems as all known projection algorithms suffer from the curse of dimensionality \cite{huynh1992practical}. Thus, the offline computation of $\Gamma_N$ can quickly become intractable as the size of the state vector, input vector, reference, or prediction horizon grows. 

\section{Theoretical Analysis} \label{ss:theoretical_properties}
This section analyzes the properties of the closed-loop system under the combined FG and MPC policy. We begin with some definitions. The feasible set of the FG is 
\begin{equation}
	\Lambda = \Gamma_N \cap (\reals^{n_x}\times \mc{R}_\epsilon).
\end{equation}
The closed-loop dynamics of \eqref{eq:LTI_system} under the combined FG and MPC feedback law are 
\begin{equation} \label{eq:closed-loop-fg-system}
  x_{k+1} = f(x_k,g(x_k,r)),
\end{equation} 
and the constrained output is
\begin{equation}
	y_k = Cx_k + D \kappa(x_k,v_{k}).
\end{equation}

\subsection{Safety and Recursive Feasibility}
The following theorem provides sufficient conditions under which the Feasibility Governor (FG) is recursively feasible and achieves the \textit{Safety} objective.

\begin{thm}[Safety \& Recursive Feasibility]\label{thm:rec_feas}
Given Assumptions~\ref{ass:stabilizable}--\ref{ass:DARE}, consider the closed-loop dynamics \eqref{eq:closed-loop-fg-system}. Given $x_0 \in \proj{x}\Lambda$, then  $(x_k,v_k)\in\Lambda$ and $y_k \in \mc{Y}$, $\forall k \in\ints$.
\end{thm}
\begin{proof}
The proof is by induction. Given $x_0\in \proj{x}\Lambda$, the FG optimization problem is feasible at time $k=0$ and $(x_0,v_0)\in \Lambda$. Next, given $(x_k,v_k)\in \Lambda$, Theorem~\ref{thm:MPC_stab} implies $(f(x_k,v_k),v_k) = (x_{k+1},v_k) \in \Lambda$. Moreover, since $(x_{k+1},v_k) \in \Lambda$ it follows that \eqref{eq:fg_opt_problem} is feasible at timestep $k+1$ and thus $(x_{k+1},v_{k+1}) = (x_{k+1},g(x_{k+1},r)) \in \Lambda$. Therefore, by induction, $(x_k,v_k)\in \Lambda \subset \Gamma_N,~\forall k\in\ints$. Finally, $y_k \in \mc{Y},~\forall k\in\ints$ follows from $(x_k,v_k)\in \Gamma_N$ and Theorem~\ref{thm:MPC_stab}.
\end{proof}

\subsection{Convergence and Stability}
Having established recursive feasibility, we now consider convergence and stability, starting with asymptotic stability. Throughout the section the reference is constant so we suppress all dependencies on $r$ to simplify the notation. 

The proof is via the invariance principle \cite[Theorem 2]{hurt1967some}, so we begin by introducing the Lyapunov function candidate
\begin{equation} \label{eq:lyapunov_injective}
	V(v) = \|v - r\|_2^2,
\end{equation}
and the increment function
\begin{equation}
	\Delta V(x,v) = V(g(f(x,v))) - V(v).
\end{equation}
The first step is to characterize the set
\begin{equation}
	\Omega = \{(x,v)\in \Lambda~|~\Delta V(x,v) = 0\}.
\end{equation}
\begin{lmm} \label{lmm:increment_root}
Under Assumptions~\ref{ass:stabilizable}--\ref{ass:DARE}, there exists $m>0$ such that
\begin{equation}
	\Delta V(x,v) \leq - m\|g(f(x,v)) - v\|^2 \leq 0,
\end{equation}
for all $(x,v)\in \Lambda$ and thus
\begin{align*}
	\Omega &= \{(x,v)\in\Lambda~|~g(f(x,v)) = v\}.
\end{align*}
\end{lmm}
\begin{proof}
The function $V$ is strongly convex and continuously differentiable. Thus, there exists $m > 0$ such that
\begin{equation*}
	V(v) \geq V(v') + \nabla V(v')^T (v - v') + m\|v - v'\|_2^2
\end{equation*}
for all $v',v \in \reals^{n_v}$. Letting $x^+ = f(x,v)$ and $v' = g(x^+) \in S_v(\Lambda,x^+)$, we have that
\begin{equation}
	\nabla V(v')^T(v - v') \geq 0,~~\forall v \in S_v(\Lambda,x^+).
\end{equation}
By using $v' = g(f(x,v))$, and recalling that $(x,v)\in \Lambda$ implies $(x^+,v) \in \Lambda$ we obtain
\begin{equation*}
 	V(g(f(x,v))) - V(v) \leq - m\|g(f(x,v)) - v\|^2 \leq 0,
\end{equation*} 
for all $(x,v)\in \Lambda$. Thus $\Delta V(x,v) \leq 0$ for all $ (x,v)\in \Lambda$ and
$\Omega = \{(x,v)\in \Lambda~|~g(f(x,v)) = v\}$.
\end{proof}

The next Lemma shows that, if $x$ is close enough to $\bar x_v$, then the FG is able to make progress towards $r^\star$.
\begin{lmm}\label{lmm:convergence}
Given Assumptions \ref{ass:stabilizable}--\ref{ass:DARE}, define
\begin{equation}\label{eq:ss_tube}
    \mc{B}_\delta(\Sigma)=\{(x,v)~|~v\in \mc{R}_\epsilon,~\|x- G_x v\| \leq\delta\},
\end{equation}
where $\Sigma = \{(x,v)~|~x = G_x v,~v\in \mc{R}_\epsilon\}$ is the equilibrium manifold. Then, there exists $\delta^\star>0$ such that $g(x) \neq v$ for all $(x,v)\in\mc{B}_{\delta}(\Sigma)$, $v\neq r^\star$ and $\delta\in [0,\delta^\star]$.
\end{lmm}

\begin{proof}
Lemma \ref{lmm:strictly_ss} implies that $\mc{B}_0(\Sigma) = \Sigma \subset\Int\Gamma_N$ hence there exists $\delta^\star>0$ such that $\mc{B}_{\delta}(\Sigma)\subset \Int\Gamma_N$ for all $\delta \in [0,\delta^\star]$. Since $\mc{B}_{\delta}(\Sigma)\subset \Int\Gamma_N$, for any $(x,v)\in \mc{B}_{\delta}(\Sigma)$ there exists $\alpha = \alpha(\delta) > 0$ such that $(x,v')\in \Gamma_N$ for all $v'\in \mc{B}_\alpha(v)$, or, equivalently, $\mc{B}_\alpha(v) \subseteq S_v(\Gamma_N,x)$.

The goal is to show that there exists $v' \in S_v(\Lambda,x)$ such that $V(v') < V(v)$ whenever $v\neq r^\star$. First, fix any $\delta \in [0,\delta^\star]$ and $\alpha = \alpha(\delta)$,  define the set $\mc{C}_\alpha = \mc{R}_\epsilon \cap \mc{B}_\alpha(v)$ and the ray $v'(t) = v + t(r^\star - v)$ with $t\geq 0$, and assume $v\neq r^\star$. It is evident that $v'(t) \in \mc{C}_\alpha$ for $t\in \left [0, \min\left(1,\beta\right)\right]$, with $\beta = \frac{\alpha}{\|v-r^\star\|}$, because $v$ and $r^\star$ are in the convex set $\mc{R}_\epsilon$ and $\|v'(t)-v\| \leq \alpha$ for all $t \leq \beta$. Further, $V$ is strongly convex and $\min_{s\in\mc{R}_\epsilon}V(s) = V(r^\star) \leq V(v)$ for all $v\in \mc{R}_\epsilon$, implying that
\begin{align*}
	V(v'(\beta)) &= V((1-\beta)v + \beta r^\star)\\
	         &< V(v) + \beta[V(r^\star) - V(v)]\\
	         &< V(v)
\end{align*}
for all $v\in \mc{R}_\epsilon \setminus r^\star$. Since $v'(\beta) \in \mc{C}_\alpha$ we have 
\begin{equation} \label{eq:lmm31}
	\min_{s\in \mc{C}_\alpha}~V(s) \leq V(v'(\beta)) < V(v),
\end{equation}
and, since $\mc{C}_\alpha \subseteq S_v(\Lambda,x) = \mc{R}_\epsilon \cap S_v(\Gamma_N,x)$,
\begin{equation} \label{eq:lmm32}
	V(g(x)) = \min_{s\in S_v(\Lambda,x)}~V(s)  \leq \min_{s\in \mc{C}_\alpha}~V(s).
\end{equation}
Combining \eqref{eq:lmm31} and \eqref{eq:lmm32} we conclude
\begin{equation}
	V(g(x)) < V(v)\quad \forall (x,v) \in \mc{B}_\alpha(\Sigma),~v\neq r^\star
\end{equation}
and thus, because $V$ is strongly convex, $g(x)\neq v$ for all $(x,v) \in \mc{B}_\delta(\Sigma),~v\neq r^\star$, and $\delta \in [0,\delta^*]$.
\end{proof}

Lemma~\ref{lmm:invariant_injective} characterizes the largest invariant set in $\Omega$, an essential step in the application of the invariance principle.
\begin{lmm} \label{lmm:invariant_injective}
Let Assumptions~\ref{ass:stabilizable}--\ref{ass:DARE} hold. Then the point
\begin{equation}
	\mc{I}_0 = \left\{(\bar x^\star_{r},r^\star)\right\},~~\bar x^\star_{r}= G_x r^\star
\end{equation}
is the largest invariant set in $\Omega \subset \Lambda$.
\end{lmm}
\begin{proof}
Let $\mc{I}\subset \Omega$ be the largest invariant set in $\Omega$. Our approach is to erode $\mc{I}$ in several steps.  

First, define the function $\phi : \mathbb{N}_{\geq 0} \times \Gamma_N \to \reals^{n_x}$ such that $\phi(\ell,x,v)$ denotes the solution of $x_{k+1} = f(x_k,v)$ starting from $x_0 = x$ at timestep $\ell \geq 0$. Let
\begin{equation*}
	\tilde \Omega = \{(x,v)\in \Lambda~|~g(\phi(\ell,x,v)) = v,~\forall \ell \geq 0\} \subseteq \Omega,
\end{equation*}
be the set of all initial conditions for which $v$ remains constant for all time. Since the set
\begin{equation*}
	\Omega \setminus \tilde \Omega = \{(x,v)\in \Lambda~|~\exists \ell > 0,~~ g(\phi(\ell,x,v)) \neq v\},
\end{equation*}
cannot be invariant, $\mc{I}\subseteq \tilde \Omega$. To see that $\Omega \setminus \tilde \Omega$ isn't invariant, consider any initial condition $(x_0,v_{-1})\in \Omega \setminus \tilde \Omega$ and consider the point $(x_{\ell-1},v_{\ell-1})$ on the resulting trajectory. By the definition of $\Omega \setminus \tilde \Omega$, $g(x_\ell,v_{\ell-1}) = g(f(x_{\ell-1},v_{\ell-1}),v_{\ell-1}) \neq v_{\ell-1}$, which implies $(x_{\ell-1},v_{\ell-1}) \notin \Omega$. Thus, for any $(x_0,v_{-1}) \in \Omega \setminus \tilde \Omega$, the resulting trajectory exits $\Omega \setminus \tilde \Omega$.

The inclusion $\mc{I}\subseteq \tilde \Omega$ implies that the auxiliary reference must remain constant in $\mc{I}$, so we focus on the evolution of $x$ given a constant $v$. Recall the set $\mc{B}_\delta(\Sigma)$ defined in \eqref{eq:ss_tube}. Due to Theorem~\ref{thm:MPC_stab}, for any $(x,v)\in \Lambda$ and $\delta > 0$, there exists a finite $t = t(\delta) \geq 0$ such that $(\phi(t,x,v),v) \in \mc{B}_{\delta}(\Sigma)$. Further, by Lemma~\ref{lmm:convergence}, there also exists $\delta^\star > 0$ such that $(\phi(t,x,v),v) \in \mc{B}_{\delta}(\Sigma)$ implies $g(\phi(t,x,v)) \neq v$ for any $\delta \in [0,\delta^\star]$ and $v\neq r^\star$. Thus, $\tilde\Omega \setminus \mc{B}_\delta(\Sigma)$ is not invariant and $\mc{I}\subseteq (\tilde \Omega \cap \mc{B}_\delta(\Sigma))$ for arbitrarily small $\delta > 0$. 

The inclusion $\mc{I}\subseteq (\tilde \Omega \cap \mc{B}_\delta(\Sigma))$ for arbitrarily small $\delta>0$ implies $\mc{I}\subseteq \mc{B}_0(\Sigma) = \Sigma$. 
Thanks to Lemma~\ref{lmm:strictly_ss}, $\Sigma \subset \Int \Gamma_N$, which implies
\begin{equation}
	g(G_x r^\star) = \argmin_{s\in \mc{R}_\epsilon}~V(s) = r^\star.
\end{equation}
Thus, $v = r^\star \implies g(G_xv) = v$. Next, applying Lemma~\ref{lmm:convergence} with $\delta = 0$, we have that $v\neq r^\star \implies g(G_xv) \neq v$ which is equivalent to $g(G_x v) = v \implies v = r^\star$. As a result, $g(G_xv) = v \Leftrightarrow v = r^\star$ and, therefore,
\begin{equation}
	\mc{I} = \{(x,v)~|~x = G_x v,~ v= r^\star\} = (\bar x_r^\star, r^\star),
\end{equation}
completing the proof.
\end{proof}

Having assembled all the components, we can now invoke the invariance principle to show asymptotic stability and finite-time convergence.
\begin{thm}[Asymptotic Stability] \label{thm:asymptotic_convergence}
Let Assumptions~\ref{ass:stabilizable}--\ref{ass:DARE} hold. Then, $(\bar x^\star_r, r^\star)$ is an asymptotically stable equilibrium point of the closed-loop system \eqref{eq:closed-loop-fg-system}, with domain of attraction $\mc{D} = \proj{x}\Lambda \times \reals^{n_v}$.
\end{thm}
\begin{proof}
Consider the candidate Lyapunov function $V:\mc{R}_\epsilon \to \reals$ defined in \eqref{eq:lyapunov_injective} and note that $V$ is continuous, bounded below, and, $V(v_{k+1}) \leq V(v_k)$ for all $(x_k,v_k)\in \Lambda$ by Lemma~\ref{lmm:increment_root}. Moreover, $x_0 \in \proj{x} \Lambda \implies (x_k,v_k) \in \Lambda$ for all $k \geq 0$ (Theorem~\ref{thm:rec_feas}), thus \eqref{eq:closed-loop-fg-system} is Lyapunov stable. Further, invoking Lemma~\ref{lmm:invariant_injective}, the largest invariant subset of $\Omega = \{(x,v)~|~ \Delta V(x,v) = 0\}$ is $(\bar x^\star_r,r^\star)$. Therefore, by the invariance principle \cite[Theorem 2]{hurt1967some}, $(x_k,v_k) \to (\bar x^\star_r,r^\star)$ as $k\to \infty$ for all $x_0 \in \proj{x}\Lambda$.
\end{proof}

\begin{thm}[Finite-time Convergence] \label{thm:finite_time_convergence}
Let Assumptions~\ref{ass:stabilizable}--\ref{ass:DARE} hold and consider the closed-loop system \eqref{eq:closed-loop-fg-system}. Then, $\forall x_0\in \proj{x}\Lambda$, there exists $t \geq 0$ such that $v_k = r^\star,~\forall k \geq t$.
\end{thm}
\begin{proof}
Due to Lemma~\ref{lmm:strictly_ss}, $(\bar x^\star_r,r^\star)\in \Sigma \subset \Int \Gamma_N$. Thus, $\bar x^\star_r \in \Int S_x(\Gamma_N,r^\star)$. In addition, the definition of $\Lambda = \Gamma_N \cap (\reals^{n_x} \times \mc{R}_\epsilon)$ implies $S_x(\Lambda,v) = S_x(\Gamma_N,v)$ for all $v \in \mc{R}_\epsilon$ and therefore $\bar x^\star_r \in \Int S_x(\Lambda,r^\star)$. Since $\bar x_r^\star \in \Int S_x(\Lambda,r^\star)$ and $x_k \to \bar x^\star_r$ as $k\to \infty$ (Theorem~\ref{thm:asymptotic_convergence}) there exists a finite $t \geq 0$ such that $x_t \in S_x(\Lambda,r^\star)$. By strong convexity of $V$, it follows from \eqref{eq:fg_opt_problem} that $g(x) = r^\star$ for all $x \in S_x(\Lambda,r^\star)$, which implies $v_k = r^\star$ for all $k \geq t$.
\end{proof}

\section{Numerical Examples} \label{ss:double_integrator}
 We consider a double integrator example, which allows us to represent the geometries of the various sets. The system matrices are 
\begin{gather*}
    A = \begin{bmatrix} 1 & 0.1 \\ 0 & 1 \end{bmatrix}, 
    B = \begin{bmatrix} 0 \\ 0.1 \end{bmatrix}, 
    C = \begin{bmatrix} 1 & 0 \\ 0 & 1 \\ 0 & 0 \end{bmatrix}, 
    D = \begin{bmatrix} 0 \\ 0 \\ 1 \end{bmatrix}, 
\end{gather*}
$E = \begin{bmatrix} 1 & 0 \end{bmatrix}$, and $F = 0$, and the sampling time is $t_s=0.1$. The default constraint set is 
\begin{equation*}
	\mc{Y}_1 = [-1,~1] \times [-0.25,~0.25] \times [-0.25,~0.25],
	\label{eq:const1}
\end{equation*}
and the MPC parameters are $Q = I$, $R = 1$, and $N = 10$ unless otherwise specified. The reference $r = 0.75$ and initial condition $x_0 = [-1,~0]^T$ are chosen such that $r\in\mc{R}_{0.01}$ and $x_0 \notin S_x(\Gamma_{10},r)$. For all the following figures, the terminal set $\mathcal{T} = \tilde{O}_{\infty}^{0.01}$ is computed using the procedure in \cite{gilbert1991linear}.

Figure~\ref{fig:gamma_vs_T} illustrates the sets $\mathcal{T}$ and $\Gamma_{10}$. The terminal set $\mathcal{T}$ is entirely contained in the feasible set, and in both cases $v$ is implicitly bounded by the constraints on $x_1$. The trajectory of the double integrator is displayed in Figure \ref{fig:3D_traj} and illustrates how the MPC + FG navigates $\Gamma_{10}$. By the time $v_k=r$, the current state $x_k$ of the system has entered $\mathcal{D}_{MPC}$ (yellow). From here, the FG holds the auxiliary reference constant and the MPC controller ensures that $x_k \to \bar{x}_r$ as $k\to\infty$.

\begin{figure}[h]
	\centering
	\includegraphics[width = \columnwidth]{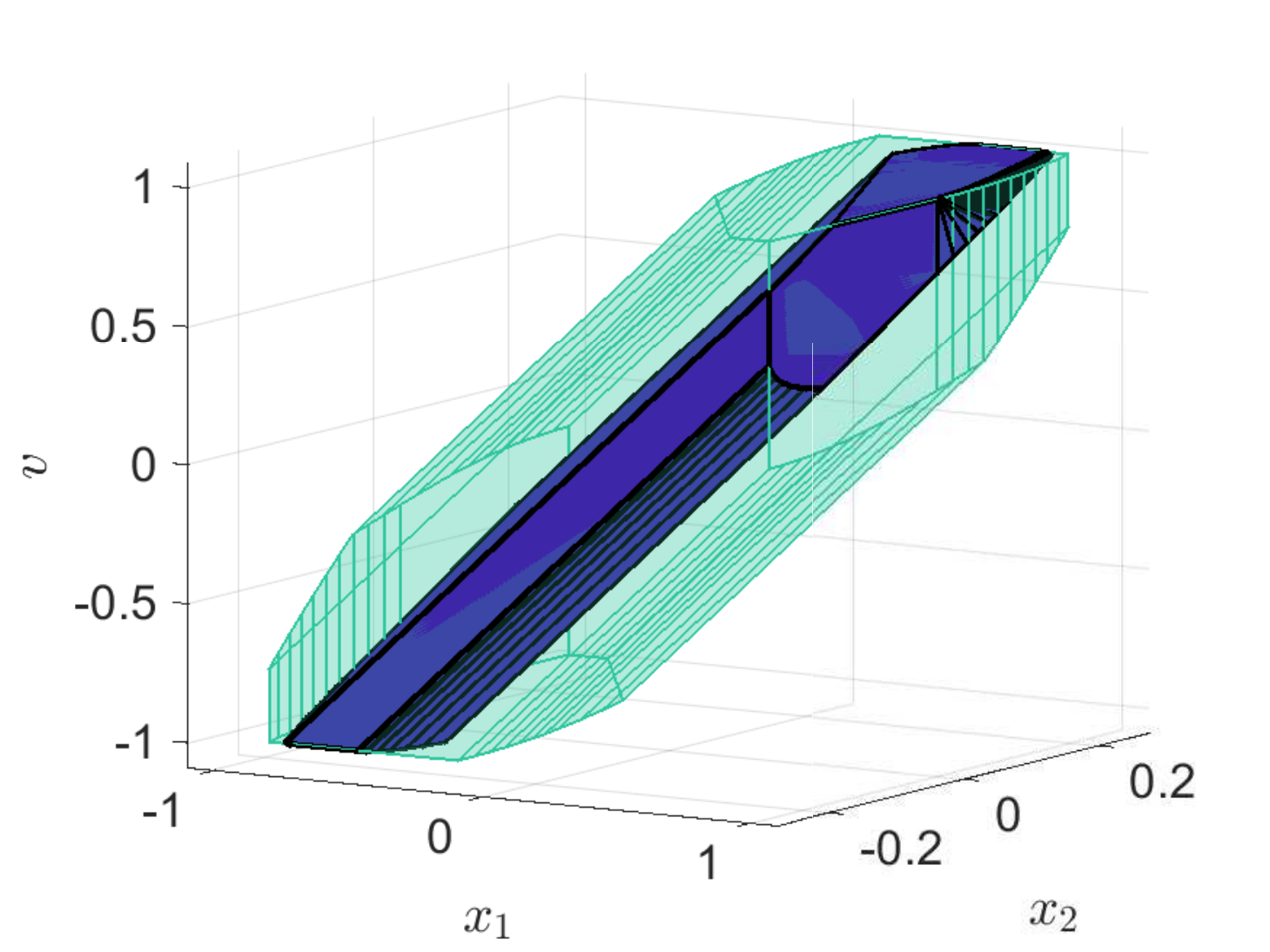}
	\caption{Terminal set $\mathcal{T}$ (blue) encased in feasible set $\Gamma_{10}$ (teal) for the double integrator with constraints $\mathcal{Y}_1$.}
	\label{fig:gamma_vs_T}
\end{figure}

\begin{figure}[h]
	\centering
	\includegraphics[width = \columnwidth]{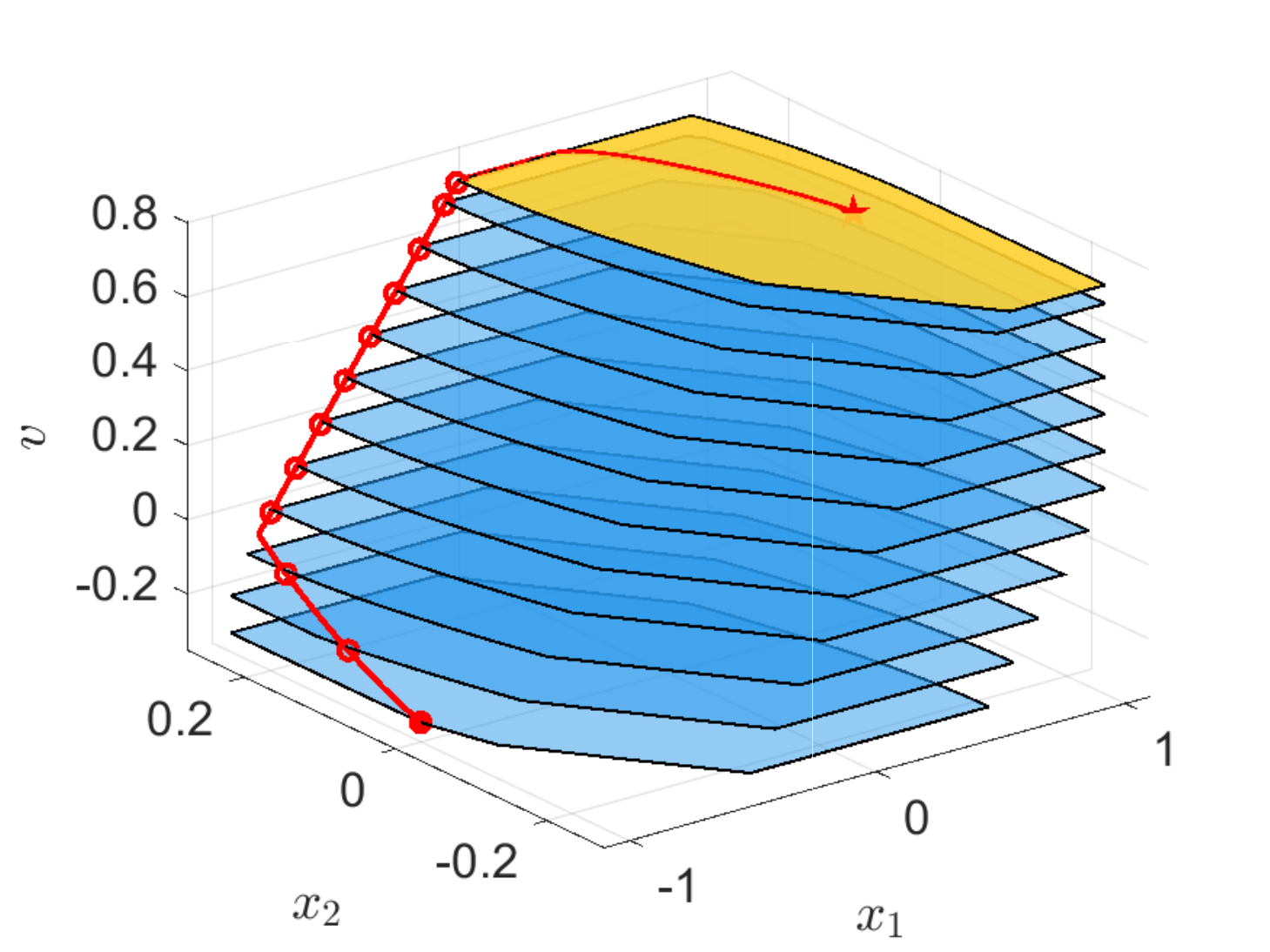}
	\caption{A closed-loop trajectory of the double integrator over the slices $S_x(\Gamma_{10},v)$ for different values of $v$ with constraints $\mathcal{Y}_1$. Circle markers show when the trajectory enters each slice and the star is the point $(x^*_r,v^*_r)$.}
	\label{fig:3D_traj}
\end{figure}

The feasible sets form an increasing sequence of sets in $N$, i.e., $\Gamma_N \subseteq \Gamma_{N+1}$ for all $N\geq 0$. This is illustrated in Figure~\ref{fig:DOA_vs_N} which uses a modified constraint set
\begin{equation*}
	\mc{Y}_2 = [-1,~1] \times [-1,~1] \times [-0.05,~0.05]
\end{equation*}
for clarity. The set $\Gamma_N$ appears to be approaching some $\Gamma_\infty \supseteq \Gamma_N$, we hypothesize that this occurs whenever $\mc{Y}$ is compact. 

\begin{figure}
	\centering
	\includegraphics[width = \columnwidth]{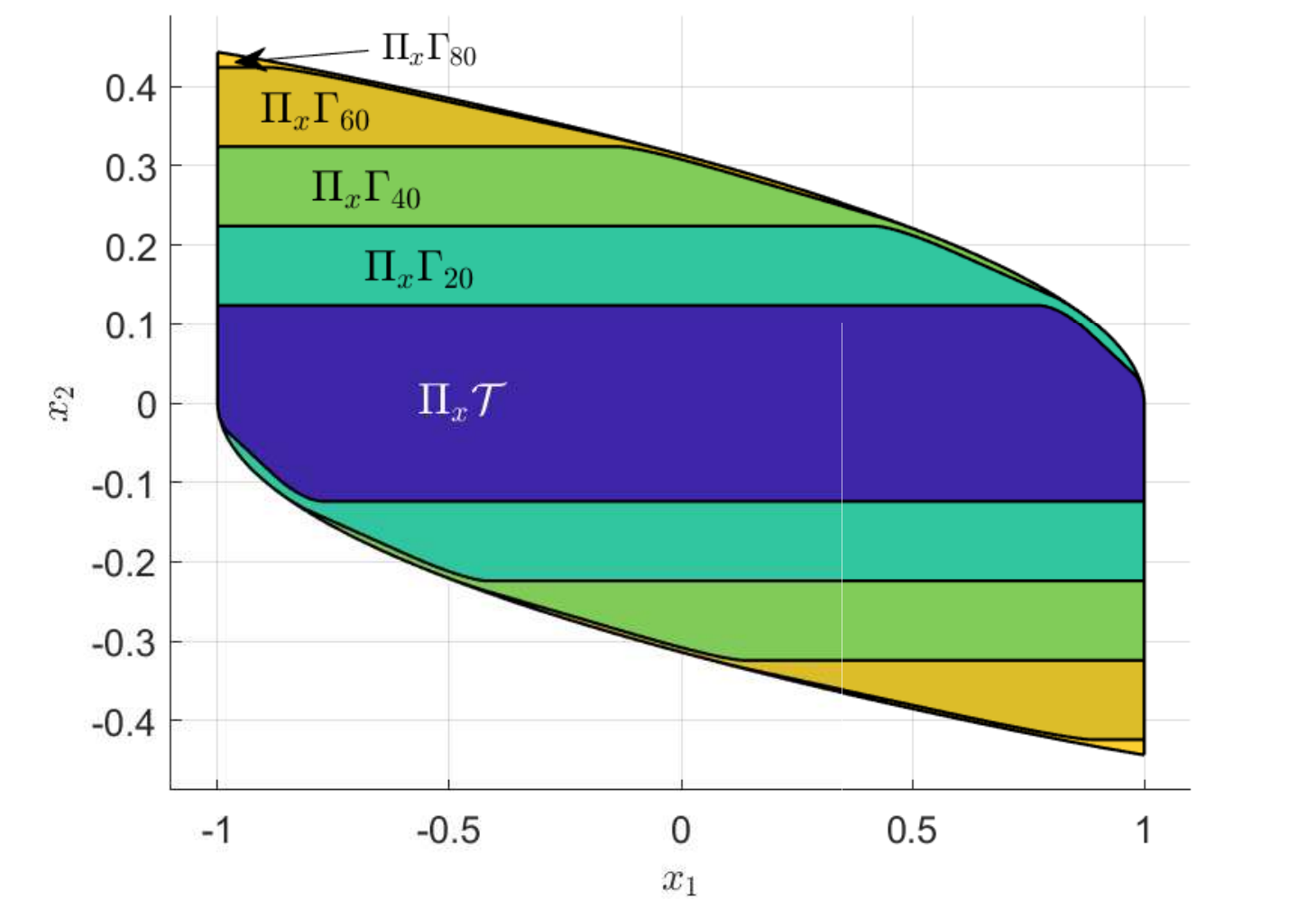}
	\caption{Increasing the control horizon $N$ expands the size of the feasible set while the terminal set stays constant. Here $\mathcal{T} = \Gamma_0 = \tilde{O}_{\infty}^{0.01}$ with constraints $\mathcal{Y}_2$ }
	\label{fig:DOA_vs_N}
\end{figure}

Figure \ref{fig:di_comparison} compares the MPC + FG feedback law with $N = 10$ to an un-goverened MPC controller with $N = N^* = 236$ where
\begin{equation}\label{eq:Nstar}
	N^* = N^*(x_0,r,\mc{T}) = \inf_i~\{i~|~(x_0,r)\in \Gamma_i\}
\end{equation}
is the smallest horizon length such that the MPC policy is feasible for the chosen $x_0$. Both these control laws are also compared to a CG applied to the LQR gain. All three controllers use $Q=100I$ and $R=1$. The constraint set
\begin{equation*}
	\mc{Y}_3 = [-20,~20] \times [-1,~1] \times [-0.25,~0.25]
\end{equation*}
is chosen to illustrate what happens when the initial conditions $x_0 = [-17,~0]^T$, and reference $r=4$ are chosen far away from each other. As displayed in Figure~\ref{fig:di_comparison}, there is $\approx$ 37\% increase in rise time using the FG, but the worst case computation time for the combined FG and MPC feedback policy is over 5000 times faster than of the un-governed MPC, as seen in Table \ref{tab:texe_di}.

\begin{figure}
	\centering
	\includegraphics[width = \columnwidth]{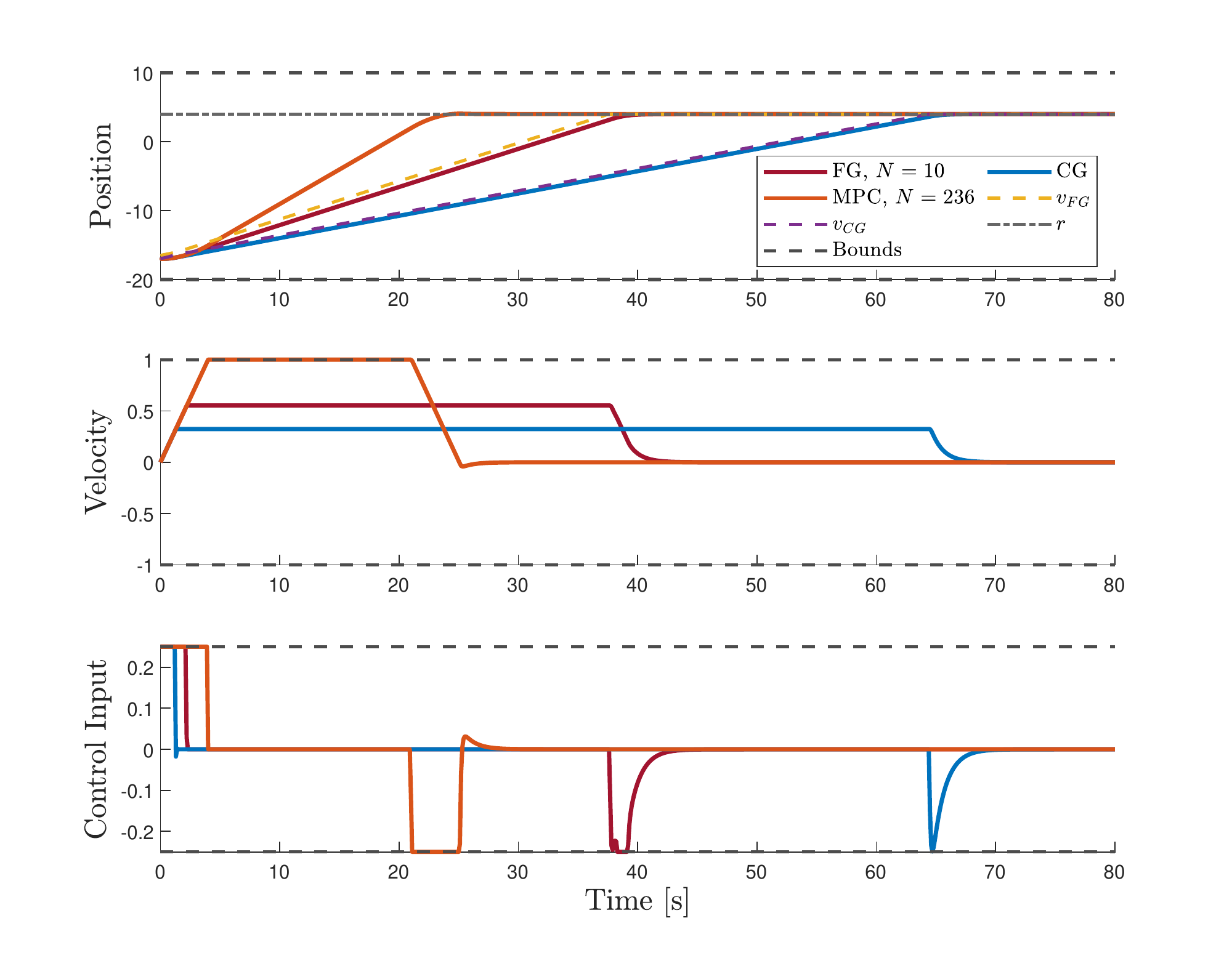}
	\caption{Closed-loop double integrator dynamics for various control laws with constraints $\mathcal{Y}_3$. The FG out performs the CG, and although the MPC has the best performance, its control horizon is too large for real-time applications.}
	\label{fig:di_comparison}
\end{figure}

\begin{table}
\centering
\caption{Execution time data for the double integrator example.}
\label{tab:texe_di}
\resizebox{\columnwidth}{!}{
\begin{tabular}{|c|c|c|c|c|} \hline
 & FG ($N = 10$) & MPC ($N =10$) & MPC ($N = 236$) & CG \\\hline
 TAVE [ms] & $0.0126$ & $0.0884$ & $255$ & $0.00833$ \\ \hline
 TMAX [ms] & $0.063$ & $0.345$ & $2170$ & $0.0369$ \\ \hline
\end{tabular}}
\end{table}

\section{Conclusions}
This paper introduced the Feasibility Governor (FG), an add-on unit that expands the region of attraction of linear model predictive controllers by filtering the reference input passed to the controller and is designed to interfere minimally with the operation and construction of the nominal controller. It was shown that the FG is safe, converges in finite time, and extends the region of attraction of MPC controllers at a fraction of the computation cost associated with increasing the prediction horizon. Future work includes extending FG theory to handle the case when $G_z$ is not invertible and introducing an easier to compute approximation of the feasible set $\Gamma_N$ to address the curse of dimensionality for larger systems.

\bibliography{feasibility-governor}

\end{document}